\documentclass[12pt, one side,emlines]{amsart}
\usepackage[a4paper, margin=1in]{geometry}
\usepackage{amssymb,latexsym,xy,eucal,mathrsfs}
\textwidth=18 cm \textheight=27cm \theoremstyle{plain}
\usepackage{lineno}
\usepackage{lipsum}
\newtheorem{theorem}{Theorem}[section]

\newtheorem{lemma}[theorem]{Lemma}

\newtheorem{definition}[theorem]{Definition}

\begin{document}
%\linenumbers
\title{On graphic splitting of regular matroids}
%\thanks{Supported by DST-SERB, Government of India, Project No. SR/S4/MS: 750/12}
\author{  Ganesh Mundhe$^1$ and   K. V. Dalvi$^2$ }
\address{ 1. Army Institute of Technology, Pune-411015, INDIA }
\address{ 2. Goverment College of Engineering, Pune-411005, INDIA}
\email{1. ganumundhe@gmail.com; 2. kvd.maths@coep.ac.in}
\maketitle
\baselineskip16truept
\begin{abstract} Raghunathan at al. \cite{raghunathan1998splitting} introduced splitting operation with respect to a pair of element for binary matroid and characterized Eulerian binary matroids using it. 
In general, the splitting operation does not preserve the graphicness property of the given matroids.  Shikare and Waphare \cite{shikare2010excluded} obtained the characterization for the class of graphic matroids which yield graphic matroids under the splitting operation with respect to a pair of elements. We study the effect of the splitting operation on regular matroids and characterize the class of regular matroids which yield graphic matroids under the splitting operation. We also provide an alternate and short proof to two of the known results. 
\end{abstract}


 \section{Introduction} 
For undefined notions and terminology, we refer to Oxley \cite{ox}. Fleischner \cite{fleischner1990eulerian} introduced the splitting operation with respect to a pair of edges of graphs. Using this operation, he characterized  Eulerian graphs and gave an algorithm to find all Eulerian trails in an Eulerian graph. 
Fleischner \cite{fleischner1990eulerian} defined the splitting operation  as follows.
\begin{definition} \cite{fleischner1990eulerian} \label{swt2g}
	Let $G$ be a connected graph and let $v$ be a vertex
	of degree at least three in $G$. If $x=vv_1$ and $y=vv_2$ are two
	edges incident at $v$, then splitting away the pair $\{x,y\}$ from $v$
	results in a new graph $G_{x,y}$ obtained from $G$ by deleting the
	edges $x,y$ and adding a new vertex $v_{x,y}$ adjacent to $v_1$
	and $v_2$. The transition from $G$ to $G_{x,y}$ is called the
	splitting operation on the graph $G$ with respect to $x$ and $y$.
\end{definition}

As an extension of the splitting operation to binary matroids,  Raghunathan, Shikare and Waphare \cite{raghunathan1998splitting} defined the splitting operation for binary matroids with respect to a pair of elements as follows.  

\begin{definition}\label{cds11} \cite{raghunathan1998splitting}  Let $M$ be a binary matroid with  standard matrix representation $A$ over the field $GF(2)$ and let $\{x,y\} \subset E(M).$  Let $A_{x,y}$ be the matrix  obtained by adjoining one extra row to the matrix $A$ whose entries are 1 in the columns labeled by the elements $x$ and $y$ and zero otherwise.  The vector matroid of the matrix $A_{x,y},$ denoted by $M_{x,y},$  is called as the splitting matroid of  $M$ with respect to $x$ and $y$, and the  transition from $M$ to $M_{x,y}$ is called as the {\it splitting operation} with respect to $x$ and $y$.
	
\end{definition}

In general, the splitting operation may not preserve some properties of the given matroid; see \cite{bm,  borse2012connected,  borse2014excluded, borse2015characterization, shikare2010excluded}.  It is an interesting to check that when the graphic matroids remain graphic under the splitting operation.  Some research in this direction is already came into lights.    In \cite{shikare2010excluded}, forbidden-minor characterization for the class of graphic matroids whose splitting matroids are graphic is obtained. In fact, they proved the following result about the graphic splitting of graphic matroids.

 \begin{theorem}\cite{shikare2010excluded} \label{sw} The splitting operation, by any pair of elements, on a graphic matroid $M$ is graphic if and only if $M$ has no minor isomorphic to any of the circuit matroids $M(G_1), M(G_2), M(G_3)$ and $M(G_4),$ where $G_1, G_2, G_3$ and $ G_4$ are the graphs as shown in Figure 1. 
 \end{theorem}
 
  \begin{center}
%TeXCAD (http://texcad.sf.net/) Picture. File: [c2 graph to graph Shikare.pic]. Options on following lines.
%\grade{\on}
%\emlines{\off}
%\epic{\off}
%\beziermacro{\on}
%\reduce{\on}
%\snapping{\off}
%\pvinsert{% Your \input, \def, etc. here}
%\quality{8.000}
%\graddiff{0.005}
%\snapasp{1}
%\zoom{4.0000}
\unitlength 1mm % = 2.845pt
\linethickness{0.4pt}
\ifx\plotpoint\undefined\newsavebox{\plotpoint}\fi % GNUPLOT compatibility
\begin{picture}(138,44.87)(0,0)
\put(.81,19.37){\circle*{1.33}}
\put(15.81,19.37){\circle*{1.33}}
\put(15.81,32.37){\circle*{1.33}}
\put(25.81,32.37){\circle*{1.33}}
\put(25.81,19.37){\circle*{1.33}}
\put(41.81,19.37){\circle*{1.33}}
\put(41.81,32.37){\circle*{1.33}}
\put(33.81,43.37){\circle*{1.33}}
\put(51.14,32.37){\circle*{1.33}}
\put(67.48,32.37){\circle*{1.33}}
\put(67.48,19.37){\circle*{1.33}}
\put(51.14,19.37){\circle*{1.33}}
\put(59.14,43.37){\circle*{1.33}}
\put(77.48,32.37){\circle*{1.33}}
\put(77.48,19.37){\circle*{1.33}}
\put(94.81,32.37){\circle*{1.33}}
\put(94.81,19.37){\circle*{1.33}}
\put(86.14,43.37){\circle*{1.33}}
\put(77.48,26.04){\circle*{1.33}}
\put(94.81,26.04){\circle*{1.33}}
\put(15.81,32.37){\line(0,-1){13}}
\put(15.81,19.37){\line(-1,0){15}}
\put(25.81,32.37){\line(1,0){16}}
\put(41.81,32.37){\line(0,-1){13}}
\put(41.81,19.37){\line(-1,0){16}}
\put(25.81,19.37){\line(0,1){13}}
\put(25.81,32.37){\line(3,4){8.33}}
\put(33.81,43.37){\line(3,-4){8.33}}
\put(51.14,32.37){\line(1,0){16.33}}
\put(67.48,32.37){\line(0,-1){13}}
\put(67.48,19.37){\line(-1,0){16.33}}
\put(51.14,19.37){\line(0,1){13}}
\put(51.14,32.37){\line(3,4){8.33}}
\put(59.14,43.37){\line(3,-4){8.33}}
\put(67.48,32.37){\line(-5,-4){16.33}}
\put(67.48,19.37){\line(-5,4){16.33}}
\put(59.14,43.37){\line(-1,-3){8}}
\put(59.14,43.37){\line(1,-3){8}}
\put(33.81,43.37){\line(-1,-3){8}}
\put(77.48,32.37){\line(4,5){8.67}}
\put(86.14,43.37){\line(4,-5){8.67}}
\put(94.81,32.37){\line(0,-1){13}}
\put(94.81,19.37){\line(-1,0){17.33}}
\put(77.48,19.37){\line(0,1){13}}
\put(77.48,32.37){\line(4,-3){17.33}}
\put(86.14,43.37){\line(-1,-3){8}}
\put(77.48,26.04){\line(3,1){17.33}}
\put(94.81,26.04){\line(-1,0){17.33}}
\qbezier(.81,19.37)(1.31,20)(.81,19.12)
\qbezier(.81,19.12)(.69,20.12)(1.06,19.12)
\qbezier(1.06,19.12)(.19,20.37)(.81,19.62)
\put(.31,31.87){\circle*{1.12}}
\put(.31,32.37){\line(1,0){16}}
%\emline(16.31,32.37)(15.56,32.62)
\multiput(16.31,32.37)(-.09375,.03125){8}{\line(-1,0){.09375}}
%\end
\put(15.56,32.62){\line(1,0){.25}}
%\emline(15.56,19.62)(.06,32.37)
\multiput(15.56,19.62)(-.041005291,.0337301587){378}{\line(-1,0){.041005291}}
%\end
\put(.31,32.37){\line(0,-1){12.75}}
\qbezier(.56,31.87)(8.06,44.87)(15.56,32.87)
%\emline(15.56,32.62)(.56,19.62)
\multiput(15.56,32.62)(-.0388601036,-.0336787565){386}{\line(-1,0){.0388601036}}
%\end
\qbezier(.56,19.62)(7.06,7.25)(15.56,19.37)
\qbezier(25.81,19.62)(32.31,7.62)(41.81,19.62)
\put(7.31,15.55){\makebox(0,0)[cc]{}}
\put(32.81,15.55){\makebox(0,0)[cc]{}}
\put(43.31,24.12){\makebox(0,0)[cc]{}}
\put(63.56,39.37){\makebox(0,0)[cc]{}}
\put(55.31,40.37){\makebox(0,0)[cc]{}}
\put(56.06,36.12){\makebox(0,0)[cc]{}}
\put(60.56,36.12){\makebox(0,0)[cc]{}}
\put(88.31,22.87){\makebox(0,0)[cc]{}}
\put(85.31,16.87){\makebox(0,0)[cc]{}}
\put(75.81,22.12){\makebox(0,0)[cc]{}}
\put(7.75,40.12){\makebox(0,0)[cc]{}}
\put(60.5,.75){\makebox(0,0)[cc]{Figure 1}}
\put(9.25,8.25){\makebox(0,0)[cc]{$G_1$}}
\put(35,8.25){\makebox(0,0)[cc]{$G_2$}}
\put(58.75,8.25){\makebox(0,0)[cc]{$G_3$}}
\put(85.5,8.25){\makebox(0,0)[cc]{$G_4$}}
\put(101.915,25.62){\circle*{1.33}}
\put(110.495,32.37){\circle*{1.33}}
\put(126.835,32.37){\circle*{1.33}}
\put(137.335,25.12){\circle*{1.33}}
\put(126.835,19.37){\circle*{1.33}}
\put(110.495,19.37){\circle*{1.33}}
\put(110.495,32.37){\line(1,0){16.33}}
\put(126.835,32.37){\line(0,-1){13}}
\put(126.835,19.37){\line(-1,0){16.33}}
\put(126.835,32.37){\line(-5,-4){16.33}}
\put(126.835,19.37){\line(-5,4){16.33}}
%\emline(101.355,25.75)(137.355,25.5)
\multiput(101.355,25.75)(4.5,-.03125){8}{\line(1,0){4.5}}
%\end
%\emline(126.605,32.25)(137.355,25)
\multiput(126.605,32.25)(.05,-.03372093){215}{\line(1,0){.05}}
%\end
%\emline(137.355,25)(126.855,19.25)
\multiput(137.355,25)(-.061403509,-.033625731){171}{\line(-1,0){.061403509}}
%\end
%\emline(110.605,32.75)(101.605,25.75)
\multiput(110.605,32.75)(-.043269231,-.033653846){208}{\line(-1,0){.043269231}}
%\end
%\emline(101.355,25.75)(110.605,19.25)
\multiput(101.355,25.75)(.047927461,-.033678756){193}{\line(1,0){.047927461}}
%\end
\put(119.105,14.5){\makebox(0,0)[cc]{$G_5$}}
\end{picture}

  \end{center}

We observe that $M(G_1)$ is a minor of $M(G_4)$. Therefore there are only three forbidden-minors $M(G_1)$, $M(G_2)$ and $M(G_3)$ for the class of graphic matroids which yield graphic matroids under the splitting operation.

Borse et al. \cite{borse2014excluded}  obtained the following  characterization for the class of cographic matroids which give cographic matroids. 
 \begin{theorem} \cite{borse2014excluded} \label{c2bsd} The splitting operation, by any pair of elements, on a cographic matroid $M$ is cographic if and only if $M$ has no minor isomorphic to any of the circuit matroids $M(G_1)$ and $ M(G_2),$  where $G_1$ and $G_2$ are the graphs as shown in Figure 1.
 \end{theorem}

Similar characterization is obtained by Borse et al. \cite{borse2015characterization}, they  characterized graphic matroids which yield cographic matroids under the splitting operation with respect to a pair of elements.  Naiyer \cite{pirouz2015graphic} obtained such a  characterization of the class of cographic matroids which yield graphic matroids.

In this paper, we  characterize regular matroids which yield graphic matroids under the splitting operation with respect to a pair of elements.   Further, we give alternate shorter proofs to two known results about getting cographic (graphic) splitting matroids from regular (cographic) matroids with respect to  a pair of elements.

In the next section, we provide  a shorter and alternate proof of a known result about a characterization of  cographic matroids giving graphic splitting matroids with respect to a pair of elements.   In Section 3, we characterize regular matroids which yield graphic (or cographic)  matroids under the splitting operation with respect to a pair of elements.  

\section{Graphic or Cographic Splitting}

In this section, we discuss  a forbidden-minor characterization for class of graphic (cographic) matroids which yield cographic (graphic) matroids under the splitting operation with respect to a pair of elements. There are some obvious forbidden-minors for these classes.    Naiyer \cite{pirouz2015graphic} obtained a  characterization of the class of cographic matroids which yield graphic matroids under the splitting with respect to a pair of elements.  We give an alternate short proof of this result. 

We use the following results. 
\begin{theorem} \cite{ox}. \label{c1crm} 
	A binary matroid is regular if and only if it has no minor isomorphic to $F_7$ or  $F_7^*.$ 
\end{theorem} 
\begin{theorem}\cite{ox} \label{c1cgm}
	A binary matroid is graphic if and only if it has no minor isomorphic to $F_7$, $F_7^*$, $M^*(K_{3,3})$ or $M^*(K_5)$. 
	
\end{theorem} 

\begin{theorem}\cite{ox} \label{c1ccgm}
	A binary matroid is cographic if and only if it has no minor isomorphic to $F_7$, $F_7^*$, $M(K_{3,3})$ or $M(K_5)$.
\end{theorem}

 First, we discuss  some trivial forbidden-minors for the class of graphic (cographic) matroids which yield cographic (graphic) matroids under the splitting operation with respect to a pair of elements.  Suppose $M$ is a binary matroid and let $ x$ and $y$ be elements of $M.$  The following statements follow from the definition of the splitting matroid $M_{x, y}.$ 

(i)  If $x$ or $y$ is a coloop of $M$, then both $x$ and $y$ are coloops of $M_{x,y};$

(ii) If none of $x$ and $y$ is a coloop of $M$, then  $\{x,y\}$ is a 2-cocircuit of $M_{x,y}$; 

(iii) If $\{x,y\}$ is a 2-cocircuit of $M$, then $M_{x,y}=M.$

Using the above observations, we introduce the following notation. 	
\noindent 
\vskip.3cm
{\bf Notation:} Given a binary matroid $M$, let $\widetilde{M}$ be the collection of binary matroids $N$ containing a pair of elements $ x$ and $y$ satisfying one of the following conditions. 
\begin{enumerate}
	\item $ N \backslash \{x, y\} = M$; 
	\item $\{x, y\}$ is a 2-cocircuit of $N$ and  $ N/ x = M$; 
	\item  $\{x, y\}$ is a 2-cocircuit of $N$ and $ N/ \{x, y\} = M.$ 
\end{enumerate}

Therefore if $N \in \widetilde{M}$, then $N\backslash \{x,y\}=M$ or $N/x=M$ or $N/\{x,y\}=M$ for some $x, y \in E(N)$.

In the following result, we prove that the splitting with respect to some pair of elements of any matroid beloging to the class $\widetilde{F}$ contains a minor $F$. 
\begin{lemma} \label{c4triv}
	Let $M$ be  a binary matroid containing a minor belonging to the class $\widetilde{F}$ for some matroid $F$. Then the splitting matroid $M_{x,y}$ contains $F$ as a minor for some $x, y \in E(M)$.
\end{lemma}
\begin{proof} Suppose $M$ contains a minor $N$  belonging to the class $ \widetilde{F}$. Then there exists $T_1, T_2 \subset E(M)$ such that $M\backslash T_1/T_2 =N$.

	(i)	Suppose $N$ is a coextension of $F$ by an element $x$ of $F$, so that $\{x,y\}$ is a 2-cocircuit of $N$ for some $y \in E(F)$ or $N$ is coextension of $F$ by two elements $x,y$ such that $\{x,y\}$ is a 2-cocircuit of $N$.   Then $|\{x,y\} \cap(T_1 \cup T_2)|= \phi$ and   $N_{x,y}= N$. Also, it is easy to see that $M_{x,y}\backslash T_1/T_2 =(M\backslash T_1/T_2)_{x,y}=N_{x,y}=N$.  Therefore $M_{x,y}$ contains  $F$ as a minor. 
	
	(ii) Suppose $N$ is an extension of $F$ by two elements, say  $x$ and $y$.  
	Then $N\backslash \{x,y\}\cong F$. Also, $M_{x,y} \backslash T_1/T_2 \backslash \{x,y\} = N_{x,y} \backslash \{x,y\} = N\backslash \{x,y\}=F$.	Therefore $M_{x,y}$ contains  $F$ as a minor. 	
\end{proof}

By Theorem \ref{c1ccgm}, a graphic matroid $M$ is not cographic if it contains $M(K_5)$ or $M(K_{3,3})$. 
Hence by  Lemma \ref{c4triv}, the members of the classes $ \widetilde M( {K_5})$ and $ \widetilde M( {K_{3,3}})$  are the forbidden-minors for the class of graphic matroids which yield cographic matroids under the splitting operation with respect to a pair of elements. 
Similarly, the classes $\widetilde M^*({K_5})$ and $ \widetilde M^*({K_{3,3}})$ contain matroids which are the forbidden-minors for the class of cographic matroids which yield graphic matroids under the splitting operation with respect to a pair of elements. 

Suppose $M$ is graphic and $M_{x,y}$ is not cographic for some $x,y \in E(M)$. Then the following lemma proves that $M$ contains a special type of minor. 

\begin{lemma}\label{c2pairgtocg}
	Let $M$ be a graphic matroid.  If $M_{x,y}$ is not cographic for some $x,y \in E(M)$, then one of the following holds.
	\begin{enumerate}
		\item  There is a minor $N$ of $M$ such that $N \in \widetilde M( {K_5})$ or $N \in \widetilde M( {K_{3,3}})$.
		\item There is a minor $N$ of $M$ containing $x$ and $y$ and avoiding 2-cocircuit such that $N_{x,y}/x \cong F$ or $N_{x,y}/\{x,y\} \cong F$ for some   $ F \in \{F_7, F^*_7, M(K_5),M(K_{3,3})\}$. 
		
	\end{enumerate}
\end{lemma}
\begin{proof} Let $ F \in \{F_7, F^*_7, M(K_5), M(K_{3,3})\}$. Therefore $F$ is a 3-connected matroid.   Assume that $M_{x,y}$ is not cographic.  Then, by Theorem \ref{c1ccgm},  $M_{x,y}$ has a minor isomorphic to $F.$ Therefore $M_{x,y} \backslash T_1/T_2 \cong F$ for some $ T_1, T_2 \subset E(M).$  Let $ T_i' = \{x,y\} \cap T_i$ and $ T_i'' = T_i - T_i'$ for $ i = 1, 2.$ Then $T_i'$ is a subset of $\{x,y\}$ while $T_i''$ is disjoint from $\{x,y\}.$  Then  $M_{x,y} \backslash T_1''/T_2'' = (M \backslash T_1''/T_2'')_{x,y}.$ Let $N = M \backslash T_1''/T_2''.$ Then $N$ is a minor of $M$  containing $\{x,y\}$ such that $ F \cong M_{x,y} \backslash T_1 /T_2 = N_{x,y}\backslash T_1'/T_2'.$
	Since $N$ is a minor of the graphic matroid $M$, $N$ is also graphic.

	Suppose $|T_2'|= \phi$. Then $F\cong N_{x,y} \backslash T_1'$. If $|T_1'|=\phi$, then $N_{x,y} \cong F$. Since $\{x,y\}$ contains a cocircuit of  $N_{x,y}$, but $F$ does not contain a loop or a  2-cocircuit, we get a contradiction.  If $|T_1'|=1$, say $T_1'=\{x\}$, then $y$ is a coloop of $N_{x,y}\backslash T_1'$ and $N_{x,y}\backslash x \cong F$, again a contradiction.  Hence  $|T'_1|=2$. Then $T_1'=\{x,y\}$ and  $N_{x,y}\backslash T_1'=N_{x,y}\backslash \{x,y\} =N\backslash \{x,y\}\cong F$. In this  case, $N$ is an extension of $F$ by $\{x,y\}$. Since $M$ is graphic, $F\neq F_7$ or $F\neq F^*_7$. Hence $F=M(K_5)$ or $M(K_{3,3})$.  Therefore $N \in \widetilde{M}(K_5)$ or $N \in \widetilde{M}(K_{3,3})$.  	 
	
	Suppose $|T_1'|=\phi$.   Then $T_2'\neq \emptyset$ and $N_{x,y}/T_2' \cong F$. If $T_2'=\{x,y\}$, then $N_{x,y}/\{x,y\} \cong F$.  Suppoe $|T_2'|=1$. Then $T_2'=\{x\}$ or $T_2'=\{y\}$. Since $\{x,y\}$ is a 2-cocircuit of $N$ or both $x$ and $y$ are coloops of $N$, we have  $N_{x,y}/x\cong N_{x,y}/y$. Hence we may assume that $T_2'=\{x\}$ and so $N_{x,y}/x\cong F$.

	Now, we prove that if $N_{x,y}/x \cong F$ or $N_{x,y}/\{x,y\} \cong F$, then $N$ does not contain a 2-cocircuit.
	On the contrary, assume that $\{x_1,x_2\}$ is a cocircuit of $N$. Then $\{x_1,x_2\}$ is a cocircuit of $N_{x,y}$. If $\{x,y\} \cap \{x_1,x_2\}=\emptyset$ then $N_{x,y} \backslash T_1' /T_2'$  contains $\{x_1,x_2\}$ as a cocircuit, a contradiction to the fact that $F$ does not contain 2-cocircuit. Assume that $|\{x,y\} \cap \{x_1,x_2\}|=1$. Suppose $x=x_1$. Therefore $N_{x,y}/x_1 =N_{x,y}/x_2$. In this case, we can replace $N$ by $N/x_2$. Hence $\{x_1,x_2\}=\{x,y\}.$  Then $N_{x,y}=N$ and so $N \ncong F$.  Therefore $F \cong N_{x,y}/x=N/x$ or $F\cong N_{x,y} /\{x,y\}=N /\{x,y\}$. Hence $N$ contains a minor belonging to the class $\widetilde{M}(K_5)$ or $\widetilde{M}(K_{3,3})$. 
\end{proof}

Similarly, if $M$ is a cographic matroid such that $M_{x,y}$ is not graphic for some $x,y \in E(M)$, then  $M$ contains a special type of minor. 
\begin{lemma}\label{2paircgtog}
	Let $M$ be a cographic matroid.  If $M_{x,y}$ is not graphic for some $x,y \in E(M)$, then one of the following holds.
	\begin{enumerate}
		\item  There is a minor $N$ of $M$ such that $N\in \widetilde M^*( {K_5})$ or $N \in \widetilde M^*( {K_{3,3}})$.
		\item There is a  minor $N$ of $M$ containing $x$ and $y$ and avoiding 2-cocircuit  such that $N_{x,y}/x \cong F$ or $N_{x,y}/\{x,y\} \cong F$ for some 
		$ F \in \{F_7, F^*_7, M^*(K_5), M^*(K_{3,3})\}$. 
		
	\end{enumerate}
\end{lemma}

The minor $N$ of $M$ in condition (1) of Lemma \ref{c2pairgtocg} and Lemma \ref{2paircgtog} are \textit{trivial minor}. Therefore, we avoid such trivial forbidden-minors for the class of graphic(cographic) matroids $M$ whose splitting matroids $M_{x,y}$ are   cographic(graphic).

Borse et al.\cite{borse2015characterization}  characterized graphic matroids whose splitting matroids are cographic by assuming condition (2) of Lemma \ref{c2pairgtocg}.

\begin{theorem} \cite{borse2015characterization} \label{bsn1} Let $M$ be a graphic matroid without containing a minor belonging to  the class $\widetilde{N}$, where $N\in \{M(K_5), M(K_{3,3})\}$.  Then  $M_{x,y}$ is cographic for any $x,y \in E(M)$ if and only if  $M$ has no minor isomorphic to any of the circuit matroids $M(G_1)$,  $ M(G_2)$ and $M(G_5)$,  where $G_1$, $G_2$ and $G_5$ are the graphs as shown in Figure 1.
\end{theorem}

Naiyer \cite{pirouz2015graphic} characterized cographic matroids which give graphic splitting matroids with respect to a pair of elements in the following theorem by assuming condition (2) of Lemma \ref{2paircgtog}.  
\begin{theorem} \cite{pirouz2015graphic} \label{c2n1n} Let $M$ be a cographic matroid without containing a minor belonging to the class $\widetilde{N}$, where $N\in \{M^*(K_5), M^*(K_{3,3})\}$.  Then $M_{x,y}$ is graphic for any $x,y \in E(M)$  if and only if $M$ has no minor isomorphic to any of the circuit matroids $M(G_1)$ and  $ M(G_2)$,  where $G_1$ and $G_2$  are the graphs as shown in Figure 1.
\end{theorem}

We observe that the above theorem follows easily from the known result as stated  in Theorem \ref{sw}. Thus, we give a very short alternate proof  of the above theorem as follows.

\begin{proof} [{\bf Proof of Theorem \ref{c2n1n}}] 
	Let $M$ be a cographic matroid.  Suppose $M$ contains a minor isomorphic to $M(G_1)$ or $ M(G_2)$. Then, by Theorem \ref{sw}, $M_{x,y}$ is not graphic for some $x$ and $y$ of $E(M).$

	Conversely, suppose $M$ does not contain a minor isomorphic to  any of the matroids $M(G_1)$ and $ M(G_2)$.  Assume that $M$ is not graphic.   Then, by Theorem \ref{c1ccgm}, $M$ contains $M^*(K_5)$ or $M^*(K_{3,3})$ as a minor.  
	\begin{center}
		%TeXCAD (http://texcad.sf.net/) Picture. File: [c2 grph to cogphi proof.pic]. Options on following lines.
		%\grade{\on}
		%\emlines{\off}
		%\epic{\off}
		%\beziermacro{\on}
		%\reduce{\on}
		%\snapping{\off}
		%\pvinsert{% Your \input, \def, etc. here}
		%\quality{8.000}
		%\graddiff{0.005}
		%\snapasp{1}
		%\zoom{4.0000}
		\unitlength 1mm % = 2.845pt
		\linethickness{0.4pt}
		\ifx\plotpoint\undefined\newsavebox{\plotpoint}\fi % GNUPLOT compatibility
		\begin{picture}(59.725,43.285)(0,0)
		\put(43.06,31.62){\circle*{1.33}}
		\put(43.06,18.62){\circle*{1.33}}
		\put(59.06,18.62){\circle*{1.33}}
		\put(59.06,31.62){\circle*{1.33}}
		\put(51.06,42.62){\circle*{1.33}}
		\put(1.06,25.62){\circle*{1.33}}
		\put(9.64,32.37){\circle*{1.33}}
		\put(25.98,32.37){\circle*{1.33}}
		\put(36.48,25.12){\circle*{1.33}}
		\put(25.98,19.37){\circle*{1.33}}
		\put(9.64,19.37){\circle*{1.33}}
		\put(43.06,31.62){\line(1,0){16}}
		\put(59.06,31.62){\line(0,-1){13}}
		\put(59.06,18.62){\line(-1,0){16}}
		\put(43.06,18.62){\line(0,1){13}}
		\put(43.06,31.62){\line(3,4){8.33}}
		\put(51.06,42.62){\line(3,-4){8.33}}
		\put(9.64,32.37){\line(1,0){16.33}}
		\put(25.98,19.37){\line(-1,0){16.33}}
		\put(25.98,32.37){\line(-5,-4){16.33}}
		\put(25.98,19.37){\line(-5,4){16.33}}
		\put(35,4.75){\makebox(0,0)[cc]{Figure 2}}
		%\emline(25.75,32.25)(36.5,25)
		\multiput(25.75,32.25)(.05,-.03372093){215}{\line(1,0){.05}}
		%\end
		%\emline(36.5,25)(26,19.25)
		\multiput(36.5,25)(-.061403509,-.033625731){171}{\line(-1,0){.061403509}}
		%\end
		%\emline(9.75,32.75)(.75,25.75)
		\multiput(9.75,32.75)(-.043269231,-.033653846){208}{\line(-1,0){.043269231}}
		%\end
		%\emline(.5,25.75)(9.75,19.25)
		\multiput(.5,25.75)(.047927461,-.033678756){193}{\line(1,0){.047927461}}
		%\end
		%\emline(43,31.75)(59,18.5)
		\multiput(43,31.75)(.0407124682,-.0337150127){393}{\line(1,0){.0407124682}}
		%\end
		%\emline(43,18.75)(59.25,31.5)
		\multiput(43,18.75)(.042989418,.0337301587){378}{\line(1,0){.042989418}}
		%\end
		\put(19,12.75){\makebox(0,0)[cc]{$G_6$}}
		\put(51.25,12.75){\makebox(0,0)[cc]{$G_7$}}
		\end{picture}
		
	\end{center}
	Let  $G_6$ and $G_7$ be the graphs as shown in Figure 2. Then $M(G_6)$ and $M(G_7)$ are the minors of $M(K_{3,3})$ and $M(K_5)$, respectively. Hence $M^*(G_6)$ and $M^*(G_7)$ are the minors of $M^*(K_{3,3})$ and $M^*(K_5)$, respectively. One can easily check that  $M^*(G_6)=M(G_1)$ and $M^*(G_7)=M(G_2)$.
	Therefore $M$  contains a minor isomorphic to $M(G_1)$ or $M(G_2)$, a contradiction. 
	Thus, $M$ is graphic and cographic.  Hence $M$ does not contain $M(K_5)$ as a minor. Note that $M(G_3)=M(K_5),$ where $G_3$ is the graph as shown in Figure 1.	Therefore, by Theorem \ref{sw}, $M_{x,y}$ is graphic.  \end{proof}

\section{Splitting of Regular Matroids}
In this section, we characterize the class of \textit{regular} matroids which yield \textit{graphic} matroids under the splitting with respect to a pair.  Naiyer et al. \cite{pirouz2011excluded} obtained a characterization of  the class of \textit{regular} matroids which yield \textit{cographic} matroids under the splitting operation with respect to  a pair of elements. We provide an alternate and very short proof of this result.

We need the following well known result about regular matroids.
\begin{theorem} \cite{ox} \label{ro} Every regular matroid $M$ can be constructed by means of direct sums, $2$-sums and $3$-sums starting with matroids each of which is isomorphic to a minor of $M$, and each
	of which is either graphic, cographic, or isomorphic to $R_{10}.$
\end{theorem}

\noindent{\bf Observation:} We prove below that the circuit matroid $M(G_1)$ is a minor of the matroid $R_{10}$, where $G_1$ is the graph as shown in Figure 1. 

Let $A$ be the  standard matrix representation of $R_{10}$. Then

\vskip.5cm
\begin{center} 
	$A = $\bordermatrix{ ~&1&2&3&4&5&6&7&8&9&10\cr
		~&1&0&0&0&0&1&1&0&0&1\cr
		~&0&1&0&0&0&1&1&1&0&0\cr
		~&0&0&1&0&0&0&1&1&1&0\cr
		~&0&0&0&1&0&0&0&1&1&1\cr   
		~&0&0&0&0&1&1&0&0&1&1}. 
\end{center}
\vskip .2cm
Let $T=\{4,5\}$. Then  the standard matrix representation of $R_{10}/\{4,5\}$ is as follows

\vskip.5cm

\begin{center} 	
	$B = $\bordermatrix{ ~&1&2&3&6&7&8&9&10\cr
		~&1&0&0&1&1&0&0&1\cr
		~&0&1&0&1&1&1&0&0\cr
		~&0&0&1&0&1&1&1&0}.
\end{center}
It is easy to check that the standard matrix representation of $M(G_1)$ is $B$.  Therefore $M(G_1)$ is a minor of regular matroid $R_{10}.$

By using the above observation,  we obtain a characterization of the class of regular matroids which yield graphic matroids under the splitting operation with respect to a pair of elements as follows.

\begin{theorem}  Let $M$ be a regular matroid without containing a minor belonging to the class  $\widetilde{N}$, where $N \in \{M^*(K_5), M^*(K_{3,3})\}$. 	Then the splitting matroid of $M$ with respect to any pair of elements is graphic if and only if $M$ does not contain a minor isomorphic to any of the matroids $M(G_1)$,  $ M(G_2)$ and $M(K_5)$, where $G_1$ and $G_2$ are the graphs as shown in Figure 1 and $K_5$ is the complete graph on $5$ vertices. 
\end{theorem}
\begin{proof} Note that $M(G_3)=M(K_5)$, where $G_3 $ is the graph as shown in Figure 1.  Suppose $M$ contains a minor isomorphic to $M(G_1)$,  $ M(G_2)$ or $M(K_5)$.  Then, by Theorem \ref{sw}, the splitting matroid $M_{x,y}$  is not graphic for some $x, y \in E(M)$.

	Conversely, suppose $M$ does not contain any of the matroids $M(G_1)$,  $ M(G_2)$ and $M(K_5)$ as a minor. We prove that $M_{x,y}$ is graphic for any $x$ and $y$.    Since $M$ is a regular matroid,  by Theorem \ref{ro}, $M$ is graphic or cographic or contains $R_{10}$ as a minor. 
	Suppose $R_{10}$ is a minor of $M$.  As observed before $R_{10}$ contains $M(G_1)$ as a minor. Hence $M(G_1)$ is a minor of $M$, a contradiction. Therefore $M$ is graphic or cographic. 	If $M$ is graphic then, by Theorem \ref{sw}, $M_{x,y}$ is graphic. Also,  if $M$ is cographic, then, by Theorem \ref{c2n1n},  $M_{x,y}$ is graphic. 
	Therefore $M_{x,y}$ is graphic for any $x,y \in E(M)$.
\end{proof}

Naiyer et al. \cite{pirouz2011excluded} obtained the forbidden-minors for the class of regular matroids $M$ whose splitting matroids  $M_{x,y}$  are cographic by assuming condition (2) of Lemmas \ref{c2pairgtocg} and \ref{2paircgtog} as follows. 

\begin{theorem} \cite{pirouz2011excluded} \label{n}  
	Let $M$ be a regular matroid without containing a minor belonging to the class $\widetilde{N}$, where $N\in \{M(K_5), M(K_{3,3})\}$.  Then $M_{x,y}$ is cographic for any $x,y \in E(M)$  if and only if $M$ has no minor isomorphic to any of the  matroids $M(G_1)$, $ M(G_2)$ and $M[A_1]$, where $G_1$ and $G_2$  are the graphs  in
	Figure 1 and  $A_1$ is the following matrix.
\end{theorem}
\begin{center} 
	
	$A_{1} = $\bordermatrix{ ~&~&~&~&~&~&~&~&~&~&~&~\cr
		~&1&0&0&0&0&1&1&0&0&1&0\cr
		~&0&1&0&0&0&1&1&1&0&0&0\cr
		~&0&0&1&0&0&0&1&1&1&0&0\cr
		~&0&0&0&1&0&0&0&1&1&1&0\cr   
		~&0&0&0&0&1&1&0&0&1&1&0\cr
		~&0&0&0&0&0&0&0&0&0&0&1}.
	
\end{center} 
\vskip.3cm
The class of graphic matroids is a subclass of the class of regular matroids. By Theorem \ref{bsn1}, $M(G_1)$, $M(G_2)$ and $M(G_3)$ are the forbidden-minors for the class of graphic matroids which yield cographic matroids under the splitting operation with repsect to a pair of elements. Therefore $M(G_3)$ is also the forbidden-minor for the class of regular matroids which yield cographic matroids under the splitting operation with respect to a pair of elements. However, $M(G_3)$ is missing in Theorem \ref{n}.  Further, the vector matroid $M[A_1]$ of $A_1$ contains a minor $R_{10}$ and $R_{10}$ contains $M(G_1)$ as a minor.  Therefore $M[A_1]$ is redundant in Theorem \ref{n} and so can be dropped.

In light of this discussion, we restate Theorem \ref{n} with appropriate modifications and prove it with a very short proof. 
\begin{theorem}  \label{nproof} Let $M$ be a regular matroid without containing a minor belonging to the class $\widetilde{N}$, where $N\in \{M(K_5), M(K_{3,3})\}$.  Then, $M_{x,y}$ is cographic for any $x,y \in E(M)$  if and only if $M$ has no minor isomorphic to any of the  matroids $M(G_1)$, $ M(G_2)$ and $M(G_3)$, where $G_1$,  $G_2$ and $G_3$  are the graphs as shown  in Figure 1.
\end{theorem}
\begin{proof} Suppose $M$ contains a minor isomorphic to one of the matroids $M(G_1)$,  $ M(G_2)$ and $M(G_3)$.	Then, by Theorem \ref{bsn1},  $M_{x,y}$ is not cographic for some pair $x, y$ of elements of $M$.

	Conversely, suppose  $M$ is a regular matroid without containing a minor isomorphic to any of the matroids $M(G_1)$, $ M(G_2)$ and $M(G_3)$.  By Theorem \ref{ro}, $M$ is graphic or cographic or contains $R_{10}$ as a minor. 	
	Since $R_{10}$ contains $M(G_1)$ as a minor, $M$ avoids $R_{10}$ as a minor. Therefore $M$ is graphic or cographic. 
	Suppose $M$ is a graphic matroid.  Then, by Theorem \ref{bsn1}, $M_{x,y}$ is cographic for any $x$ and $y$ of $E(M)$.  If $M$ is a cographic matroid, then, by Theorem \ref{c2bsd},  $M_{x,y}$ is cographic for any $x,y \in E(M)$. 
\end{proof}


\begin{thebibliography}{99}
  
 \bibitem{bm}{Y. M. Borse and G. Mundhe, On n-connected splitting matroids, \textit{AKCE Int. J. Graphs Comb.} \textbf{16} (1) (2019), 50-56.}
 
 \bibitem{borse2012connected} {Y. M. Borse and  S. B. Dhotre, {On connected splitting matroids,}  \textit{Southeast Asian Bull. Math.} {\bf 36} (1) (2012), 17-21.}
 
 \bibitem{borse2014excluded}  {Y. M. Borse, M. M. Shikare and K. V. Dalvi, {Excluded-minors for the class of cographic splitting matroids,} \textit{ Ars Combin.} {\bf 115} (2014), 219-237.}
 
 \bibitem{borse2015characterization} {Y. M. Borse, M. M. Shikare  and Pirouz Naiyer, {A characterization of graphic matroids which yield cographic splitting matroids,}  \textit{ Ars Combin.} {\bf 118}  (2015), 357-366.}
 
  
 \bibitem{fleischner1990eulerian} {H. Fleischner,  {\it Eulerian Graphs and Related Topics Part 1, Vol. 1,}  North Holland, Amsterdam, 1990.}
 
 
 \bibitem{pirouz2011excluded} { P. Naiyer, K. Dalvi and M. M. Shikare, {On the excluded minors for regular matroids which yield cographic splitting matroids,}  \textit{Lobachevskii J. Math.} {\bf 32}  (2011), 376-384.}
 
 \bibitem{pirouz2015graphic} { P. Naiyer, {Graphic splitting of cographic matroids,}  \textit{Discuss. Math. Graph Theory} {\bf 35} (2015), 95-104.}	
  
  \bibitem{ox} {J. G. Oxley, {\it Matroid Theory}, Second Edition, Oxford University Press, Oxford, 2011.}
  
  
 
 
 \bibitem{raghunathan1998splitting} { T. T. Raghunathan, M. M. Shikare   and  B. N. Waphare, {Splitting in a binary matroid,}   \textit{Discrete Math.} {\bf 184} (1998), 267-271.}
 
  

 %\bibitem{shikareazadi} {M. M. Shikare and G. Azadi, Determination of the bases of a splitting matroid, \textit{European J. Comb.} \textbf{24} (1) (2003), 45-52.}
 

  \bibitem{shikare2010excluded} {M. M. Shikare  and  B. N. Waphare, {Excluded-Minors for the class of graphic splitting matroids,}  \textit{ Ars Combin.} {\bf 97} (2010), 111-127.}
  
  \bibitem{slater1974classification} { P. J. Slater, {A Classification of $4$-connected graphs,}  \textit{J. Combin. Theory Ser. B} {\bf 17} (1974), 281-298.}
 
 \bibitem{tutte1961theory} {W. T. Tutte. {A theory of 3-connected graphs,}  \textit{Indag. Math.} {\bf 23}  (1961), 441-455.}
 
 \bibitem{tutte1966connectivity} {W. T. Tutte,  {Connectivity in matroids,}  \textit{ Can. J. Math.} {\bf 18}  (1966), 1301-1324.}
 
 \bibitem{west2001introduction} {D. B. West, {\it Introduction to Graph Theory}, Second Edition, Prentice Hall of India, New Delhi,  2006.}
 
 

 \end{thebibliography}
\end{document}